\documentclass[11pt]{article}
\usepackage[a4paper,margin=28mm]{geometry}
\usepackage[T1]{fontenc}
\usepackage{lmodern,amsmath,amssymb,amsthm,mathtools,microtype}
\usepackage[colorlinks=true,linkcolor=blue,citecolor=blue,urlcolor=blue]{hyperref}
\numberwithin{equation}{section}
\newtheorem{theorem}{Theorem}[section]
\newtheorem{lemma}[theorem]{Lemma}
\newtheorem{proposition}[theorem]{Proposition}
\newtheorem{corollary}[theorem]{Corollary}
\theoremstyle{definition}
\newtheorem{example}[theorem]{Example}
\newcommand{\R}{\mathbb R}
\newcommand{\dA}{\dim_{\mathrm A}}
\newcommand{\dH}{\dim_{\mathrm H}}
\newcommand{\dP}{\dim_{\mathrm P}}
\newcommand{\Gr}{\operatorname{Gr}}
\newcommand{\rank}{\operatorname{rank}}
\title{Orthogonal Projections and Assouad Dimension: Resolving Fraser’s Questions}
 \author{Bilel Selmi\thanks{Corresponding author. Analysis, Probability and Fractals Laboratory LR18ES17, Department of Mathematics, Faculty of Sciences of Monastir, University of Monastir, 5000 Monastir, Tunisia. Email: \texttt{bilel.selmi@fsm.rnu.tn}.}}
\begin{document}
\maketitle
\begin{abstract}
We examine two questions on exceptional orthogonal projections for Assouad dimension. Let $1\le k<d$ and $1\le q<d$. For a set $F$ contained in a $q$-dimensional subspace $W\subset\R^d$ and containing a nonempty relatively open subset of $W$, we compute both the Hausdorff and packing dimensions of the exceptional $k$-planes as $k(d-k)-|q-k|-1$. In particular, a square in $\R^3$ has a one-dimensional set of exceptional plane projections. This disproves the bound $k(d-1-k)$ in Question 17.8.1 of Fraser's book as stated, for every $2\le k<d$. The planar packing question, Question 17.8.2, is already answered affirmatively by Wu's countability theorem. Using that theorem, we also describe the distinction between planar sets of Assouad dimension at most $1$ and those of dimension greater than $1$ when they are embedded in higher-dimensional spaces. The computations use elementary Grassmannian geometry; we make no claim of a new planar projection theorem or of an optimal bound for arbitrary higher-dimensional sets.
\end{abstract}

\section{Introduction}
Let $\Gr(k,d)$ denote the Grassmannian of $k$-dimensional linear subspaces of $\R^d$, and let $P_V$ be orthogonal projection onto $V\in\Gr(k,d)$. We use the metric $\|P_V-P_{V'}\|$, where the norm is the operator norm. The Grassmannian is a compact smooth manifold of dimension $k(d-k)$ and carries the rotation-invariant Borel probability measure $\gamma_{d,k}$. This is $G(k,d)$ in \cite{FraserBook} and $G(d,k)$ in several other references; throughout this note the first argument denotes the dimension of the subspace. For a nonempty set $F\subset\R^d$, define its exceptional set by
\begin{equation}\label{eq:exceptional}
\mathcal E_k(F)=\{V\in\Gr(k,d):\dA P_VF<\min\{k,\dA F\}\}.
\end{equation}
Fraser's Question 17.8.1 \cite[Section 17.8]{FraserBook}, repeated in the update list \cite{FraserUpdates}, asks whether
\begin{equation}\label{eq:question}
\dH\mathcal E_k(F)\le k(d-1-k)\qquad(1\le k<d)
\end{equation}
for every nonempty $F\subset\R^d$. The motivation is a set lying on a line $L$: all $V\subset L^\perp$ collapse this set to a point, and these subspaces form $\Gr(k,d-1)$. In that example the projected set is a singleton because $V$ is \emph{orthogonal} to $L$. Mere transversality to $L$ does not imply collapse.

The line example does not account for higher ranks. If $F$ is contained in a $k$-plane and has nonempty relative interior there, loss of even one rank gives a projection of dimension less than $k$. Theorem~\ref{thm:flat} computes the dimension of the entire set of such projections. It equals $k(d-k)-1$, which exceeds the right-hand side of \eqref{eq:question} by $k-1$ when $k\ge2$.

Question 17.8.2 asks whether $\dP\mathcal E_1(F)=0$ for every nonempty planar $F$. Orponen \cite{Orponen2021} proved that this exceptional set has Hausdorff dimension $0$. Subsequently, Wu \cite[Theorem 1.15]{Wu2025} proved that it is at most countable. Thus the packing question has an affirmative answer by an existing theorem. Section~\ref{sec:planar} records this consequence and applies it to sets contained in a two-dimensional subspace. Wu's theorem is an external premise here, rather than a result proved anew.

The general lower bound $\dA P_VF\ge\min\{k,\dA F\}$ for $\gamma_{d,k}$-almost every $V$ is established in \cite[Theorem 2.9]{Fraser2018}; see also \cite[Theorem 10.2.1]{FraserBook}. General exceptional-set estimates obtained from weak tangents are discussed in \cite[Theorem 2.10]{Fraser2018} and \cite[Section 2]{Fraser2023}. These almost-everywhere assertions allow positive-dimensional exceptional sets. The more recent projection results for the Assouad spectrum \cite{FFK2026} concern a different dimension and do not remove the rank obstruction in \eqref{eq:question}.

This note keeps the original Assouad dimension as the source and image dimension. The packing dimension in Question 17.8.2 measures a set of directions, not the dimension of the projected set. We give complete elementary proofs of the flat computations and the deductions from Wu's theorem. The higher-dimensional line-projection case of \eqref{eq:question}, and an optimal replacement for general $k$, are not resolved here.

\section{Rank Stratification of the Grassmannian}\label{sec:ranks}
For a set $E$ in Euclidean space, let $N_r(E)$ be the least number of balls of radius $r$ covering $E$. The Assouad dimension of $E$ is defined by
\[
\dA E=\inf\bigl\{s\ge0:\exists C>0\ \forall x\in E\ \forall\,0<r<R,
\ N_r(E\cap B(x,R))\le C(R/r)^s\bigr\}.
\]
We use this definition also for unbounded sets, with all positive scales $r<R$. For bounded sets it agrees with the definition restricted to sufficiently small $R$. Assouad dimension is monotone under inclusion and invariant under bi-Lipschitz maps; it is not in general nonincreasing under Lipschitz maps \cite[Chapter 2]{FraserBook}. Hausdorff and packing dimensions below are the usual dimensions for the specified metric, with both dimensions of the empty set taken to be $0$.

Fix $1\le k<d$ and a $q$-dimensional subspace $W\subset\R^d$, where $1\le q\le d$. For $0\le j\le\min\{q,k\}$, define
\[
\mathcal S_j(W)=\{V\in\Gr(k,d):\rank(P_V|_W)=j\}.
\]
The kernel of $P_V|_W$ is $W\cap V^\perp$. In particular, we have 
\begin{equation}\label{eq:kernel}
V\in\mathcal S_j(W)\quad\text{if and only if}\quad
\dim(W\cap V^\perp)=q-j.
\end{equation}

\begin{lemma}\label{lem:strata}
The stratum $\mathcal S_j(W)$ is nonempty exactly when
\begin{equation}\label{eq:feasible}
\max\{0,q+k-d\}\le j\le\min\{q,k\}.
\end{equation}
For each such $j$, it is a smooth manifold of dimension
\begin{equation}\label{eq:stratum-dimension}
\dH\mathcal S_j(W)=\dP\mathcal S_j(W)
=k(d-k)-(q-j)(k-j).
\end{equation}
\end{lemma}
\begin{proof}
Write $U=V^\perp$, $a=q-j$, and $b=d-k-q+j$. By \eqref{eq:kernel}, $A=U\cap W$ has dimension $a$. The image $B=P_{W^\perp}U$ has dimension $b$, since $\dim U=d-k=a+b$. Thus $0\le a\le q$ and $0\le b\le d-q$, which are exactly \eqref{eq:feasible}.

Conversely, choose $A\in\Gr(a,W)$ and $B\in\Gr(b,W^\perp)$. Here $\Gr(a,W)$ denotes the $a$-planes in $W$. For a linear map $T:B\to W\cap A^\perp$, set
\begin{equation}\label{eq:graph}
U=A\oplus\{z+Tz:z\in B\}\quad\text{and}\quad V=U^\perp.
\end{equation}
Then $\dim U=a+b=d-k$ and $U\cap W=A$, so \eqref{eq:kernel} gives $V\in\mathcal S_j(W)$. This also proves nonemptiness for every feasible $j$.

Every $U$ with $U\cap W=A$ has a unique representation \eqref{eq:graph}. Indeed, the restriction of $P_{W^\perp}$ to $U\cap A^\perp$ is a linear bijection onto $B$. Its inverse sends $z$ to $z+Tz$, determining $T$ uniquely. In local Grassmannian charts, smooth orthonormal frames for $A$, $W\cap A^\perp$, and $B$ identify these maps $T$ with real matrices of size $j$ by $b$. Formula \eqref{eq:graph} is smooth in these coordinates. Conversely, the intersection and image subspaces and this inverse depend smoothly on $U$ on any chart where the corresponding fixed-rank minors are invertible. Thus the stratum is the total space of a smooth vector bundle over
\[
\Gr(a,W)\times\Gr(b,W^\perp),
\]
whose fiber has dimension $jb$. Its dimension is
\[
a(q-a)+b(d-q-b)+jb=aj+bk=k(d-k)-(q-j)(k-j).
\]
This description also applies when $a=0$ or $b=0$; the corresponding Grassmannian or matrix space is a point.

On compact subsets of a smooth coordinate chart, the chart and its inverse are Lipschitz for the Grassmannian metric. A countable cover by such subsets gives the upper bounds in \eqref{eq:stratum-dimension}, by countable stability of Hausdorff and packing dimensions. A nonempty coordinate ball contains a bi-Lipschitz image of a Euclidean ball of the stated dimension, giving the matching lower bounds. In dimension $0$, the nonempty stratum is locally discrete and has both dimensions $0$.
\end{proof}

\begin{lemma}\label{lem:interior}
Let $F\subset W$ contain a nonempty relatively open subset of $W$. For every $V\in\Gr(k,d)$, we have 
\[
\dA F=q\quad\text{and}\quad\dA P_VF=\rank(P_V|_W).
\]
\end{lemma}
\begin{proof}
A subset of a $q$-plane has Assouad dimension at most $q$: covering a ball by a regular grid gives $N_r(E\cap B(x,R))\le C_q(R/r)^q$ for $0<r<R$. A ball contained in $F$ gives the reverse bound. In fact, volume comparison gives $N_r(B_W(x,R))\ge (R/r)^q$ whenever $B_W(x,R)\subset F$, so no exponent less than $q$ satisfies the definition.

Put $j=\rank(P_V|_W)$. The image $P_VF$ is contained in the $j$-plane $P_VW$, which gives $\dA P_VF\le j$. For $j>0$, let $\sigma>0$ be the smallest singular value of $P_V|_{W\cap\ker(P_V|_W)^\perp}$. If $B_W(x,R)\subset F$, then
\[
B_{P_VW}(P_Vx,\sigma R)\subset P_V(B_W(x,R))\subset P_VF.
\]
The volume argument in dimension $j$ gives $\dA P_VF\ge j$. If $j=0$, the nonempty image is a singleton.
\end{proof}

\section{Flat Sets and the Limitations of the Proposed Bound}\label{sec:flat}
\begin{theorem}\label{thm:flat}
Let $1\le k<d$, $1\le q<d$, and let $F\subset W$, where $W$ is a $q$-dimensional subspace of $\R^d$. Suppose that $F$ contains a nonempty relatively open subset of $W$. Then
\begin{equation}\label{eq:flat}
\dH\mathcal E_k(F)=\dP\mathcal E_k(F)=k(d-k)-|q-k|-1.
\end{equation}
If instead $q=d$ and $F$ contains a nonempty open subset of $\R^d$, then $\mathcal E_k(F)=\varnothing$.
\end{theorem}
\begin{proof}
Set $m=\min\{q,k\}$ and $j_0=\max\{0,q+k-d\}$. By Lemma~\ref{lem:interior} and \eqref{eq:exceptional}, we have
\begin{equation}\label{eq:union}
\mathcal E_k(F)=\bigcup_{j=j_0}^{m-1}\mathcal S_j(W).
\end{equation}
For $q<d$, we have $j_0\le m-1$: if $q\le k$, this follows from $k\le d-1$, and if $k\le q$, it follows from $q\le d-1$. Thus the union is nonempty. The dimension in \eqref{eq:stratum-dimension} strictly increases with $j$ in this range, since for $j<m$,
\[
\bigl[k(d-k)-(q-j-1)(k-j-1)\bigr]
-\bigl[k(d-k)-(q-j)(k-j)\bigr]=q+k-2j-1>0.
\]
Finite the stability of both dimensions and \eqref{eq:union} therefore give
\[
\dH\mathcal E_k(F)=\dP\mathcal E_k(F)
=k(d-k)-(q-m+1)(k-m+1).
\]
One of $q-m$ and $k-m$ is $0$, so the last product is $|q-k|+1$, proving \eqref{eq:flat}. If $q=d$, the only feasible rank in \eqref{eq:feasible} is $j=k$, so no projection is exceptional.
\end{proof}

\begin{corollary}\label{cor:negative}
For every $2\le k<d$, the inequality in Question 17.8.1 fails for a compact cube contained in a $k$-dimensional subspace of $\R^d$.
\end{corollary}
\begin{proof}
Apply Theorem~\ref{thm:flat} with $q=k$. It gives that
\[
\dH\mathcal E_k(F)=k(d-k)-1=k(d-1-k)+(k-1)>k(d-1-k).
\]
\end{proof}

\begin{example}\label{ex:square}
Let $F=[-1,1]^2\times\{0\}\subset\R^3$ and $W=\R^2\times\{0\}$. Write a plane $V$ as $n^\perp$, where $n\in S^2$ is a unit normal, with $n$ and $-n$ representing the same plane. If $n\notin W$, then $W\cap\R n=\{0\}$, so $P_V|_W$ has rank $2$ and $\dA P_VF=2$. If $n\in W$, then its rank is $1$, and $P_VF$ is a nondegenerate line segment. Therefore, we have
\[
\mathcal E_2(F)=\{n^\perp:n\in W\cap S^2\}.
\]
This is a projective circle, with Hausdorff and packing dimensions $1$. The proposed bound is $2(3-1-2)=0$.
\end{example}

For $q=1$, formula \eqref{eq:flat} reduces to $k(d-1-k)$, recovering the motivating line example. For $q=k$, the largest exceptional stratum has codimension $1$ in the full Grassmannian. Consequently any bound valid for \emph{all} nonempty source sets and fixed $d,k$ must be at least $k(d-k)-1$. This necessary lower limit on a possible bound does not prove a general upper bound of that size. At the full-dimensional endpoint $q=d$, the exceptional set is empty, rather than the value obtained by formally extending \eqref{eq:flat} outside its range.

\section{Planar exceptions and planar sets in higher dimensions}\label{sec:planar}
We use the following theorem with its full set-theoretic scope.
\begin{theorem}[Wu \cite{Wu2025}, Theorem 1.15]\label{thm:wu}
For every nonempty set $F\subset\R^2$, the set
\[
\{L\in\Gr(1,2):\dA P_LF<\min\{1,\dA F\}\}
\]
is at most countable.
\end{theorem}
No compactness, Borel assumption, or regularity condition on $F$ is imposed in this theorem. The proof in \cite[Section 5.2]{Wu2025} passes from a set to a microscopic limit supporting a measure with ergodic scenery, and uses the countable-exception theorem for projections of that scenery. We invoke the set theorem directly; no such measure hypotheses are added to our statements.

\begin{corollary}\label{cor:packing}
For every nonempty $F\subset\R^2$, $\dP\mathcal E_1(F)=0$.
\end{corollary}
\begin{proof}
If the exceptional set is nonempty, write it as a finite or countable union of singletons. Each singleton has packing dimension $0$, and packing dimension is countably stable. Theorem~\ref{thm:wu} gives the conclusion; the empty case follows from our convention. Thus Question 17.8.2 is answered by Wu's theorem.
\end{proof}

\begin{proposition}\label{prop:planar-lift}
Let $d\ge2$, $1\le k<d$, and let $F$ be a nonempty subset of a two-dimensional subspace $W\subset\R^d$. Put $s=\dA F$.
\begin{enumerate}
\item\label{item:thin} If $k=1$ or $0<s\le1$, then
\[
\dH\mathcal E_k(F)\le k(d-1-k)\quad\text{and}\quad
\dP\mathcal E_k(F)\le k(d-1-k).
\]
If $s=0$, then $\mathcal E_k(F)=\varnothing$.
\item\label{item:thick} If $k\ge2$ and $1<s\le2$, then
\[
\mathcal E_k(F)=\mathcal S_0(W)\cup\mathcal S_1(W)\quad\text{and}\quad
\dH\mathcal E_k(F)=\dP\mathcal E_k(F)=k(d-1-k)+1,
\]
where an infeasible rank stratum is omitted.
\end{enumerate}
\end{proposition}
\begin{proof}
If $\rank(P_V|_W)=2$, then $P_V|_W$ is bi-Lipschitz onto its image and $\dA P_VF=s$. If its rank is $0$, then $P_VF$ is a singleton. If its rank is $1$, write $L=W\cap V^\perp$ and choose a unit vector $u\in W\cap L^\perp$. For $w\in W$,
\begin{equation}\label{eq:rank-one}
P_Vw=\langle w,u\rangle P_Vu,\qquad P_Vu\ne0.
\end{equation}
Thus the image in \eqref{eq:rank-one} is a similarity of the scalar projection of $F$ onto $\R u$.

\textit{(\ref{item:thin}).} Suppose first that $s>0$. Identify $W$ isometrically with $\R^2$, and define
\[
\mathcal C=\{L\in\Gr(1,W):\dA P_{W\cap L^\perp}F<\min\{1,s\}\}.
\]
By Theorem~\ref{thm:wu}, $\mathcal C$ is at most countable. In the present parameter range $\min\{k,s\}=\min\{1,s\}$. Equation \eqref{eq:rank-one} therefore gives
\begin{equation}\label{eq:lift-cover}
\mathcal E_k(F)\subset\mathcal S_0(W)
\ \cup\!\bigcup_{L\in\mathcal C}\{V\in\Gr(k,d):V\subset L^\perp\}.
\end{equation}
Each set indexed by $L$ is a copy of $\Gr(k,d-1)$ and has both dimensions $k(d-1-k)$. The rank-$0$ set is empty if $k>d-2$; otherwise it is $\Gr(k,W^\perp)$, of dimension $k(d-2-k)$. Countable stability in \eqref{eq:lift-cover} proves the two upper bounds. If $s=0$, no nonempty projection has negative Assouad dimension, so \eqref{eq:exceptional} is empty.

\textit{(\ref{item:thick}).} Now $\min\{k,s\}=s>1$. Every image of rank at most $1$ has Assouad dimension at most $1$ and is exceptional. Rank $2$ is not exceptional by bi-Lipschitz invariance. Since $d>k\ge2$, rank $1$ is feasible by \eqref{eq:feasible} with $q=2$. Lemma~\ref{lem:strata} gives its dimension as
\[
k(d-k)-(2-1)(k-1)=k(d-1-k)+1.
\]
If rank $0$ is feasible, its dimension is smaller, so the same value holds for the union.
\end{proof}

The second part does not require relative interior or any measure regularity of $F$: only $F\subset W$ and $\dA F>1$ are used. The first part gives the original proposed bound for a restricted class of sources. It does not establish that bound for arbitrary sets in $\R^d$ when $k=1$ and $d\ge3$. Together with the flat computation, it distinguishes the known planar answer from the different rank geometry of higher-dimensional projections.

 \section*{Acknowledgments}
 The author acknowledges support from his institution, including the Analysis, Probability \& Fractals Laboratory (Grant No. LR18ES17).
\vskip5mm
\noindent{\textbf{Statements and Declarations}}
\vskip4mm
\noindent{\textbf{Competing Interests}}
\vskip3mm
The authors declare that they have no known competing financial interests or personal relationships that could have appeared to influence the work reported in this article.
\vskip4mm
\noindent{\textbf{Data Availability}}
\vskip3mm
No data was used for the research described in this article.
\bibliographystyle{plain}
\bibliography{references}
\end{document}